\documentclass{amsart}
\usepackage{epsfig}

\theoremstyle{plain}
\newtheorem{theorem}{Theorem}[section]
\newtheorem{proposition}[theorem]{Proposition}
\newtheorem{lemma}[theorem]{Lemma}
\newtheorem{corollary}[theorem]{Corollary}

\theoremstyle{definition}

\theoremstyle{remark}
\newtheorem{remark}[theorem]{Remark}

\begin{document}

\title[(Disk, Essential surface) pairs of Heegaard splittings]
{(Disk, Essential surface) pairs of Heegaard splittings that
intersect in one point}

\author{Jung Hoon Lee}
\address{School of Mathematics, KIAS\\
207-43, Cheongnyangni 2-dong, Dongdaemun-gu\\
Seoul, Korea. Tel:\,+82-2-958-3736}
 \email{jhlee@kias.re.kr}


\begin{abstract}
 We consider a Heegaard splitting $M=H_1\cup_S H_2$ of a
$3$-manifold $M$ having an essential disk $D\subset H_1$ and an
essential surface $F\subset H_2$ with $|D\cap F|=1$. (We require
that $\partial F\subset S$ when $H_2$ is a compressionbody with
$\partial_{-}H_2\ne\emptyset$.)

 Let $F$ be a genus $g$ surface with $n$ boundary components.
 From $M=H_1\cup_S H_2$, we obtain a genus $g(S)+2g+n-2$
 Heegaard splitting $M=H'_1\cup_{S'} H'_2$ by cutting $H_2$
 along $F$ and attaching $F\times I$ to $H_1$ along
 $\partial F\times I$.
 As an application, by using a theorem due to Casson and Gordon
 \cite{Casson-Gordon},
 we give examples of $3$-manifolds having two Heegaard splittings of distinct genera
 where one of the two Heegaard splittings is a strongly irreducible
 non-minimal genus splitting and it is obtained from the other by the
 above construction.
\end{abstract}

\maketitle

\section{Introduction}

Every compact $3$-manifold $M$ admits a Heegaard splitting and
there are various Heegaard splittings as the genus varies. If $g$
is the minimal genus of Heegaard splittings of $M$, then for each
$g'>g$ there exists at least one Heegaard splitting of genus $g'$
--- a splitting obtained by stabilizations.

From a Heegaard splitting, we can obtain another Heegaard
splitting of different genus which is not just a stabilization if
the original one has certain embedded surfaces that intersect in
one point. A stabilized Heegaard splitting $H_1\cup_S H_2$, which
has essential disks $D_1\subset H_1$ and $D_2\subset H_2$ with
$|D_1\cap D_2|=1$, can be destabilized and the genus goes down.

Concerning (Disk, Annulus) pairs, many people \cite{Schleimer},
\cite{Saito}, \cite{Moriah-Sedgwick}, \cite{Lee}
 considered several notions as in the following and their relations
 with other notions on Heegaard splittings.

\begin{itemize}
\item Essential disk $D\subset H_1$ and essential annulus
$A\subset H_2$ with $D\cap A=\emptyset$
 \item Essential disk $D\subset H_1$ and essential annulus
$A\subset H_2$ with $\partial D$ equal to one component of
$\partial A$
 \item Essential disk $D\subset H_1$ and essential annulus
 $A\subset H_2$ with $|D\cap A|=1$
 \item Essential disk $D\subset H_1$ and spanning annulus
$A$ in a compressionbody $H_2$ with $D\cap A=\emptyset$
 \item Essential disk $D\subset H_1$ and spanning annulus $A$ in a
compressionbody $H_2$ with $|D\cap A|=1$
\end{itemize}

 In \cite{Lee}, the author considered a Heegaard splitting
$H_1\cup_S H_2$
 having an essential disk $D\subset H_1$ and an essential
 annulus $A\subset H_2$ with $|D\cap A|=1$ and it
 was shown that such a Heegaard splitting has the disjoint curve
 property, a notion which was introduced by Thompson
 \cite{Thompson},
 and another Heegaard splitting $H'_1\cup_{S'}H'_2$ can be
 obtained from $H_1\cup_S H_2$ by removing a neighborhood of $A$
 from $H_2$ and attaching it to $H_1$. In this case the genus
 of Heegaard splitting remains unchanged.

 In this paper, we consider a Heegaard splitting $H_1\cup_S H_2$
 of a $3$-manifold $M$
 having an essential disk $D\subset H_1$ and an essential
 surface $F\subset H_2$ with $|D\cap F|=1$.
We require that $\partial F\subset S$ when $H_2$ is a
compressionbody with $\partial_{-}H_2\ne\emptyset$.
 We denote it as a
 {\it{strong $(D,F)$ pair}} for consistency of terminology with
 \cite{Moriah-Sedgwick}.
 First we show that if
 $F$ has genus $g$ and $n$ boundary components, the distance
 $d(S)$ of $H_1\cup_S H_2$ is bounded above by $2g+n$ (Theorem 2.3).

 From $H_1\cup_S H_2$ we can obtain another Heegaard splitting
 $H'_1\cup_{S'} H'_2$ by removing a neighborhood of $F$ from $H_2$
 and attaching it to $H_1$.

\begin{theorem}
Let $H_1\cup_S H_2$ be a Heegaard splitting of a $3$-manifold with
a strong $(D,F)$ pair. Let $H_1'$ be obtained from $H_1$ by
attaching $F\times I$ along $\partial F\times I$ and $H_2'$ be
obtained from $H_2$ by cutting along $F$. Then $H_1'\cup_{S'}
H_2'$ is a Heegaard splitting of genus $g(S)+2g+n-2$.
\end{theorem}

The construction of new Heegaard surface in Theorem 1.1 resembles
quite a bit the {\it {Haken sum}} in Moriah, Schleimer, and
Sedgwick's paper \cite{Moriah-Schleimer-Sedgwick}. In that paper,
they considered the Haken sum of Heegaard surface with copies of
an incompressible surface in the manifold and obtained infinitely
many distinct Heegaard splittings. Also there are related works by
Kobayashi \cite{Kobayashi}, and Lustig and Moriah
\cite{Lustig-Moriah}. However, in our case the essential surface
lives in one of the compressionbodies (Fig. 1).

\begin{figure}[h]
    \centerline{\includegraphics[width=6cm]{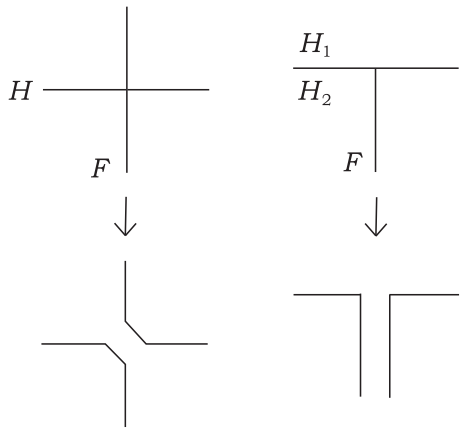}}
    \caption{Haken sum and compressing along an essential surface
    in a compressionbody}\label{fig1}
\end{figure}

In Theorem 3.5, it is shown that $H'_1\cup_{S'}H'_2$ has the
disjoint curve property if $F$ is not a disk.

In most part of the paper, we are considering the case $\partial
F\subset\partial_{+} H_2$ when $H_2$ is a compressionbody with
$\partial_{-}H_2\ne\emptyset$. In the last part of section 3, we
briefly consider the case when $F$ is a spanning annulus in a
compressionbody, and show a corresponding result (Theorem 3.6).

For $g\ge 1$ or $n\ge 3$, the genus of $H'_1\cup_{S'} H'_2$ is
greater than that of $H_1\cup_S H_2$. We give examples of
$3$-manifolds having two Heegaard splittings of distinct genera
 where one of the two Heegaard splittings is a strongly irreducible
 non-minimal genus splitting and it is obtained from the other by the
 method in Theorem 1.1.
 The examples are constructed by doing
 $1/q$-Dehn surgery ($|q|\ge 6$) on certain knots and a theorem
 due to Casson and Gordon is used to show strong irreducibility.

\begin{theorem}
Let $K$ be a knot with the following properties.
\begin{itemize}
 \item A minimal genus Seifert surface $F$ (of genus $g$) for $K$ is free.
 \item Every tunnels of an unknotting tunnel system
 $\{t_1,t_2,\cdots,t_t\}$ can be isotoped to lie on $F$ and
 mutually disjoint.
\item $\bigcup^t_{i=1} t_i$ cuts $F$ into a
 connected subsurface $F'\subset F$.
 \item $t+1<2g$
 \end{itemize}
 Let $K(1/q)$ be the manifold obtained by doing $1/q$-Dehn surgery $(|q|\ge 6)$
 on $K$ in $S^3$.
Then $K(1/q)$ has a genus $2g$ strongly irreducible Heegaard
splitting and a genus $t+1$ Heegaard splitting, and the two are
related by the construction in Theorem 1.1.
\end{theorem}

In particular, if $K$ is a torus knot, $K(1/q)$ is a
 Seifert fibered space over $S^2$ with three exceptional fibers
 \cite{Moriah-Sedgwick}, \cite{Sedgwick}, and
 Theorem 1.2 gives some insight to the relation of a vertical
 splitting and a horizontal splitting of such manifolds.

\section{Heegaard splitting with a strong $(D,F)$ pair}

 First we show that for a $(D,F)$ pair with $|D\cap F|=1$, in fact
 the essentiality of $F$ follows automatically from incompressibility. In
 other words, we have the following.

 \begin{proposition}
 For a Heegaard splitting $H_1\cup_S H_2$, let $D$ be an essential
 disk in $H_1$ and $F$ $(\partial F\ne \emptyset,\,\,
 \partial F\subset S)$ be an incompressible
 surface in $H_2$ such that $|D\cap F|=1$. Then $F$ is essential in
 $H_2$.
 \end{proposition}

\begin{proof}
Suppose $F$ is not essential in $H_2$. Then $F$ is parallel to a
subsurface $F'\subset S$ (rel. $\partial F$). When we go around
$\partial D$, we pass through $\partial F'$ from $S-F'$ to the
interior of $F'$ at some time. We should pass through $\partial
F'$ at least once more to go around all of $\partial D$. This is a
contradiction since $|D\cap F|=|D\cap
\partial F|=1$ and $\partial F=\partial F'$.
\end{proof}

Now we consider the distance, due to Hempel \cite{Hempel}, of a
Heegaard splitting with a strong $(D,F)$ pair. The {\it{distance}}
$d(S)$ of a Heegaard splitting $H_1\cup_S H_2$ is the smallest
number $n\ge 0$ so that there is a sequence of essential simple
closed curves $\alpha_0,\cdots,\alpha_n$ in $S$ with $\alpha_0$
bounding a disk in $H_1$, $\alpha_n$ bounding a disk in $H_2$ and
for each $1\le i\le n$, $\alpha_{i-1}$ and $\alpha_i$ can be
isotoped in $S$ to be disjoint.

We need the the following technical lemma on boundary compression
by Morimoto to get an upper bound for distance.

\begin{lemma}(Lemma 5.1 of \cite{Morimoto})
Let $W$ be a compact orientable $3$-manifold, and let $F$ be an
essential surface properly embedded in $W$ such that $\partial
F\ne\emptyset$ and $\partial F$ is contained in a single component
of $\partial W$. Let $F'$ be the $2$-manifold obtained from $F$ by
a boundary compression. Then $F'$ is incompressible and has a
component which is not $\partial$-parallel. Hence $F'$ is
essential.
 \end{lemma}

\begin{theorem}
Let $H_1\cup_S H_2$ be a genus $\ge 2$ Heegaard splitting with a
strong $(D,F)$ pair. If $F$ has genus $g$ and $n$ boundary
components, then the distance $d(S)\le 2g+n$.
\end{theorem}

\begin{proof}
Let $\partial F=\beta_1\cup\cdots\cup\beta_n$ and
$|D\cap\beta_1|=1$. Since $F$ is incompressible and not boundary
parallel and $\partial F\subset S$, $F$ intersects a meridian disk
system of $H_2$. By standard innermost disk and outermost arc
arguments, we may assume that there is a boundary compressing disk
$\Delta$ for $F$, where the boundary compression occurs toward
$S$. Let $\partial\Delta=\alpha\cup\beta$ where $\alpha$ is an
essential arc in $F$ and $\beta$ is an arc in $S$. We construct
sequence of essential simple closed curves
$\alpha_0,\cdots,\alpha_k$ with $\alpha_0$ bounding a disk in
$H_1$, $\alpha_k$ bounding a disk in $H_2$ and for each $1\le i\le
k$, $\alpha_{i-1}$ and $\alpha_i$ can be isotoped in $S$ to be
disjoint, dividing into two cases according to $n$.

Case $1$.\quad $n=1$.

Take two parallel copies of $D$ in $H_1$ and connect them with a
band along $\beta_1$ and push the band slightly into the interior
of $H_1$ to make a disk $D'\subset H_1$. Since $\partial D'$
bounds a once punctured torus in $S$ and $H_1\cup_S H_2$ is a
genus $\ge 2$ Heegaard splitting, $D'$ is an essential disk in
$H_1$. Note that $\partial D'$ is disjoint from $\beta_1$. Take
$\partial D'$ as $\alpha_0$ and $\beta_1$ as $\alpha_1$.

Case $2$.\quad $n>1$.

In this case, take $\partial D$ as $\alpha_0$ and any $\beta_i$
$(i\ne 1)$ as $\alpha_1$.

\vspace{0.3cm}

 Both in Case $1$ and
Case $2$, boundary compress $F$ along $\Delta$ to get an essential
surface $F_{(1)}$ by Lemma 2.2. All the boundary components of
$F_{(1)}$ can be made disjoint from $\partial F$. Take any
component of $\partial F_{(1)}$ as $\alpha_2$. Boundary compress
$F_{(1)}$ to get an essential surface $F_{(2)}$ by Lemma 2.2. All
the boundary components of $F_{(2)}$ can be made disjoint from
$\partial F_{(1)}$. Take any component of $\partial F_{(2)}$ as
$\alpha_3$. In this way, we successively boundary compress until
we get an essential disk in $H_2$ by Lemma 2.2. We can check that
the possible maximum number of boundary compressions is $2g+n-1$.
So the possible maximum length sequence of $\alpha_i$'\,s would be
$\alpha_0,\alpha_1,\cdots,\alpha_{2g+n}$. So we conclude that
$d(S)\le 2g+n$.
\end{proof}

\section{Obtaining new Heegaard splittings}

We consider attaching $F\times I$ to a handlebody along $\partial
F\times I$. Let $g(X)$ denote the genus of $X$.

\begin{lemma}
Let $\gamma_1.\cdots,\gamma_n$ be mutually disjoint loops on the
boundary of a handlebody $H$ and $D$ be an essential disk of $H$
such that $|\partial D\cap\gamma_1|=1$ and $\partial D\cap
\gamma_i=\emptyset$ $(i=2,\cdots,n)$. Let $F$ be a genus $g$
surface with $n$ $(n\ge 1)$ boundary components
$\beta_1,\cdots,\beta_n$.

If we attach $F\times I$ to $H$ along $\partial F\times I$ so that
$\beta_i\times I$ is attached to $N(\gamma_i;\partial
H)\cong\gamma_i\times I$ $(i=1,2,\cdots,n)$, then the resulting
manifold is a handlebody of genus $g(H)+2g+n-2$.
\end{lemma}

\begin{proof}
Let $p$ be the intersection point $D\cap\gamma_1$. Consider the
neighborhood $D\times I$ in $H$ and $\gamma_1\times I$ in
$\partial H$. We can assume that $\partial(D\times
I)\cap(\gamma_1\times I)$ is a small rectangle $R$ containing $p$.
Let $R'$ be the rectangle in $\beta_1\times I$ that is attached to
$R$.

Since $F$ is a genus $g$ surface with $n$ boundary components,
there are mutually disjoint essential arcs $a_1,\cdots,
a_{2g+n-1}$ in $F$ such that $F$ cut along $a_1\cup\cdots\cup
a_{2g+n-1}$ is a disk. In particular, take such an essential arc
system so as to satisfy that one of the two points of $\partial
a_1$ is attached to $p$. More precisely, we take the rectangular
parallelepiped neighborhood $a_1\times I\times I$ of $a_1$ in
$F\times I$ to be equal to $R'\times I$.

Let $H'$ be $cl(H-(D\times I))$. Since $|D\cap\gamma_1|=1$, $D$ is
a non-separating essential disk in $H$. So $H'$ is a handlebody of
genus $g(H)-1$. Attach a rectangular parallelepiped neighborhood
$a_i\times I\times I$ of $a_i$ taken in $F\times I$ to $H'$ along
$\partial a_i\times I\times I$ for each $i=2,\cdots,2g+n-1$. Since
each $a_i\times I\times I$ $(i=2,\cdots,2g+n-1)$ can be considered
as a $1$-handle, the resulting manifold $H''$ is a genus
$g(H)+2g+n-3$ handlebody.

Observe that $cl(F\times I-(\bigcup^{2g+n-1}_{i=1}a_i\times
I\times I))$ is homeomorphic to a $3$-ball $B$, which is attached
to $H''$ along two subdisks of its boundary. Then $H'''=H''\cup B$
is a handlebody of genus $g(H)+2g+n-2$. Observe also that
$(D\times I)\cup (R'\times I)$ is a $3$-ball attached to $H'''$
along $(D\times \partial I)\cup({\text{three faces of}}
\,\,\partial(R'\times I))$, which is a disk on the boundary of a
$3$-ball. So the genus remains unchanged after attaching $(D\times
I)\cup(R'\times I)$ to $H'''$. Hence we conclude that the
resulting manifold after attahcing $F\times I$ to $H$ along
$\partial F\times I$ is a genus $g(H)+2g+n-2$ handlebody.
\end{proof}

Now we consider removing a neighborhood of incompressible surface
from a compressionbody. The following lemma is well-known. It can
be found, for example, as Lemma 2 in Schulten's paper
\cite{Schultens}.

\begin{lemma}
Let $F$ $(\partial F\ne\emptyset)$ be an incompressible surface
properly embedded in a compressionbody $H$ with $\partial
F\subset\partial_{+}H$. Then $F$ cuts $H$ into
compressionbodies(or a compressionbody).
\end{lemma}



Using Lemma 3.1 and Lemma 3.2, we give a proof of Theorem 1.1.

\begin{proof}
{\it{(of Theorem 1.1.)}} \,By Lemma 3.1, $H'_1$ is a handlebody of
genus $g(S)+2g+n-2$. By Lemma 3.2, $H'_2$ is union of two
compressionbodies or a compressionbody. Since $\partial H'_1$ is
same as $\partial_{+} H'_2$, $H'_2$ is connected. Therefore
$H'_1\cup_{S'}H'_2$ is a Heegaard splitting of genus
$g(S)+2g+n-2$.
\end{proof}

\begin{corollary}
Let $H_1\cup_S H_2$ be a Heegaard splitting of a $3$-manifold $M$
with a strong $(D,A)$ pair where $A$ is an annulus. Let $H_1'$ be
obtained from $H_1$ by attaching $A\times I\subset H_2$ along
$\partial A\times I$ and $H_2'$ be obtained from $H_2$ by cutting
along $A$. Then $H_1'\cup_{S'} H_2'$ is a Heegaard splitting of
same genus with $H_1\cup_S H_2$.
\end{corollary}

\begin{remark}
Corollary 3.3 is a generalization of Definition 14 of
\cite{Schultens}.
\end{remark}

 A Heegaard splitting $H_1\cup_S H_2$ is said to have the
 {\it{disjoint curve property}} (\cite{Thompson}) if there are essential
 disks $D_1\subset H_1$, $D_2\subset H_2$ and an essential loop
 $\gamma\subset S$ such that $(\partial D_1\cup \partial D_2)\cap
 \gamma=\emptyset$. It is equivalent to that the distance $d(S)$
 is less than or equal to two.
 The newly obtained Heegaard
 splitting $H'_1\cup_{S'}H'_2$ of Theorem 1.1 has the disjoint curve
 property if the switch from $H_1\cup_S H_2$ to $H'_1\cup_{S'}H'_2$
 is not a destabilization.

\begin{theorem}
If $F$ is not a disk, the Heegaard splitting $H'_1\cup_{S'} H'_2$
obtained in Theorem 1.1 has the disjoint curve property.
\end{theorem}

\begin{figure}[h]
    \centerline{\includegraphics[width=7cm]{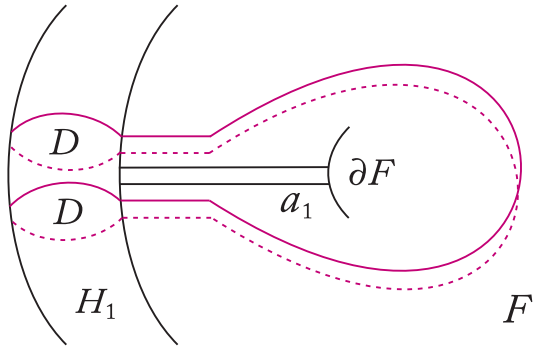}}
    \caption{An essential disk $E'$ in $H'_1$ satisfying the disjoint
    curve property}\label{fig2}
\end{figure}

\begin{proof}
Recall the proof of Lemma 3.1. In the proof of Lemma 3.1, $H'''$
was obtained from $H''$ by attaching a $1$-handle. Consider a
meridian disk (co-core) $E$ of the $1$-handle. Also remember that
$H'_1=H_1\cup (F\times I)$ was obtained from $H'''$ by attaching a
$3$-ball along a $2$-disk on its boundary. Then $E$ is enlarged to
an essential disk $E'$ in $H'_1$, which can be taken as two
parallel copies of $D$ attached by a band in $F\times I$. See Fig.
2. Since the band is equivalent to an (arc)$\times I$ in $F\times
I$ with both endpoints of the arc in the same component of
$\partial F$, we can take an essential loop $\gamma\subset F$
which is disjoint from $E'$. Since $F$ is incompressible in $H_2$,
we can see that $\gamma$ is an essential loop in the new Heegaard
surface $S'$. Take a boundary compressing disk $\Delta\subset H_2$
for $F$. Let $\partial\Delta=\alpha\cup\beta$ where $\alpha$ is an
essential arc in $F$. Then after cutting $H_2$ along $F$, $\Delta$
is an essential disk in $H'_2$. We may assume that $\alpha$
belongs to $F\times\{0\}$ and $\gamma$ belongs to $F\times\{1\}$.
So $\Delta$ is disjoint from $\gamma$. So we conclude that the
triple $(E',\Delta,\gamma)$ satisfies the disjoint curve property.
\end{proof}

Now we consider the case when $F$ is a spanning annulus in a
compressionbody. Given a compressionbody $H$ with
$\partial_{-}H\ne\emptyset$, there exists a meridian disk system
$\{D_1,\cdots,D_k\}$ of $H$ such that $H$ cut along
$\bigcup^{k}_{i=1}D_i$ is $\partial_{-}H\times I\,\,\cup$
(possibly empty) $3$-balls. A {\it spanning annulus} of a
compressionbody $H$ is an annulus which can be expressed as
$\gamma\times I$ in the compressionbody structure, where $\gamma$
is an essential loop in $\partial_{-}H$. By showing the analogue
of Lemma 3.1 and Lemma 3.2, we can state the following.

\begin{figure}[h]
    \centerline{\includegraphics[width=8cm]{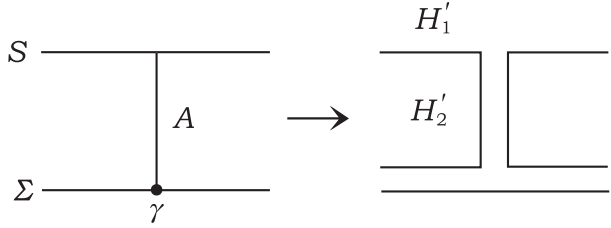}}
    \caption{Compressing along a spanning annulus}\label{fig3}
\end{figure}

\begin{theorem}
Let $H_1\cup_S H_2$ be a Heegaard splitting of a $3$-manifold $M$
with a strong $(D,A)$ pair where $A$ is a spanning annulus.
 Let $A=\gamma\times I$ in the compressionbody structure and
$\gamma\subset\Sigma\subset\partial_{-}H_2$ and $\Sigma$ has genus
$g$. Let $H_1'$ and $H_2'$ be obtained as follows (Fig. 3).

\begin{itemize}
\item $H_1'$ is obtained from $H_1$ by attaching $A\times I$ along
$(\partial A-\gamma)\times I$ and attaching $\Sigma\times I$ along
$\gamma\times I$.
 \item $H_2'$ is obtained from $H_2$ by cutting along $A$ and
 shrinking the part adjacent to $\Sigma$
\end{itemize}

 Then $H_1'\cup_{S'} H_2'$ is a Heegaard splitting of
genus $g(S)+g-1$. In particular, if $\Sigma$ is a torus, it has
same genus with $H_1\cup_S H_2$.
\end{theorem}

\begin{proof}
By following the procedure as in the proof of Lemma 3.1, we can
see that $H_1'$ is a compressionbody of genus $g(S)+g-1$. One
remarkable point is that the partition of components of $\partial
M$ is changed for the Heegaard splitting --- $\Sigma$ belongs to
$\partial_{-}H_2$ before the change and $\partial_{-}H_1'$ after
the change.

By cutting $H_2$ along $A$, $H_2'$ becomes a genus $g(S)+g-1$
compressionbody and parts of $\Sigma$ is connected to
$\partial_{+}H_2$. We conclude that $H_1'\cup_{S'} H_2'$ is a
Heegaard splitting of genus $g(S)+g-1$.
\end{proof}

\section{Examples}
 Let $K$ be a knot admitting a minimal genus free Seifert surface
 $F$ of genus $g$. Then $F$ is incompressible in $cl(S^3-N(F))$.
 Also $F$ is incompressible in the product neighborhood $N(F)=F\times I$.
 Since $N(F)$ and $cl(S^3-N(F))$ are handlebodies, this gives a Heegaard
 splitting $N(F)\cup_{\Sigma} cl(S^3-N(F))$ of $S^3$.

 Now we are going to construct a strongly irreducible Heegaard
 splitting from $\Sigma$ by Dehn surgery on $K$. Remove a
 neighborhood $N(K)$ from $S^3$. Let $K(1/q)$ denote the manifold
 obtained by $1/q$-filling on $cl(S^3-N(K))$. We can assume that the
 filling solid torus $T$ is attached to $N(F)=F\times I$ along an
 annulus $\partial F\times I$. Note that if we perform $1/q$-filling,
 a meridian curve $(1,0)$ of the filling solid torus is mapped to $(1,q)$
 curve and longitude $(0,1)$ of filling solid torus is mapped to
longitude $(0,1)$. So $N(F)\cup T$ is a handlebody. Then we get
the Heegaard splitting $(N(F)\cup T)\cup_{\Sigma'}cl(S^3-N(F))$
for $K(1/q)$. (Alternatively, we can regard $\Sigma'$ is obtained
from $\Sigma$ by Dehn twists on $K$ $|q|$ times.)
 By a theorem due to Casson and Gordon \cite{Casson-Gordon},
 $\Sigma'$ is a strongly irreducible Heegaard splitting if $|q|\ge
 6$. Here we refer the statements in (\cite{Moriah-Schultens}, Appendix).

 \begin{theorem}(Casson-Gordon)
 Suppose $M=H_1\cup_{\Sigma} H_2$ is a weakly reducible Heegaard splitting
 for the closed manifold $M$. Let $K$ be a simple closed curve in
 $\Sigma$ such that $\Sigma-N(K)$ is incompressible in both $H_1$
 and $H_2$. Then $\Sigma'$, for all $|q|\ge 6$, is a strongly
 irreducible Heegaard splitting for the Dehn filled manifold $M(1/q)$.
 \end{theorem}

\begin{proof} {\it{(of Theorem 1.2.)}}
$K(1/q)$ has a genus $2g$ Heegaard splitting $(N(F)\cup
T)\cup_{\Sigma'}cl(S^3-N(F))$. The strong irreducibility of it is
already shown above.

 Note that $cl(S^3-N(K\cup (\bigcup^{t}_{i=1} t_i)))$ is a genus
 $t+1$ handlebody and this
 handlebody remains untouched during the $1/q$-surgery on $K$.
 So $(T\cup (\bigcup^{t}_{i=1} N(t_i)))\cup_{S'} cl(S^3-N(K\cup(\bigcup^{t}_{i=1} t_i)))$
 is a genus $t+1$ Heegaard splitting for the Dehn filled manifold
$K(1/q)$.

 By assumption of Theorem 1.2, every tunnels of an unknotting tunnel system
 $\{t_1,t_2,\cdots,t_t\}$ for $K$ can be isotoped to lie on $F$ and
 mutually disjoint, and $\bigcup^t_{i=1} t_i$ cuts $F$ into a
 connected subsurface $F'\subset F$.
 Since $F'$ is a subsurface of
 $F$, $F'$ is incompressible in $cl(S^3-N(K\cup (\bigcup^{t}_{i=1}
 t_i)))$.
We can see that the meridian
 of the filling solid torus $T$ intersect $F'$ in one point since
 the determinant of the matrix
 $(\begin{smallmatrix}
 1 & 0 \\
 q & 1
 \end{smallmatrix})$
 is $1$. So the genus $t+1$ Heegaard splitting
$(T\cup (\bigcup^{t}_{i=1} N(t_i)))\cup_{S'}
cl(S^3-N(K\cup(\bigcup^{t}_{i=1} t_i)))$ has a strong $(D,F')$
pair. Since $t+1<2g$, we can see that the genus $2g$ strongly
irreducible non-minimal genus splitting $(N(F)\cup
T)\cup_{\Sigma'}cl(S^3-N(F))$ and the genus $t+1$ splitting
$(T\cup (\bigcup^{t}_{i=1} N(t_i)))\cup_{S'}
cl(S^3-N(K\cup(\bigcup^{t}_{i=1} t_i)))$ is related by the
construction as in Theorem 1.1.
\end{proof}

\begin{figure}[h]
    \centerline{\includegraphics[width=5cm]{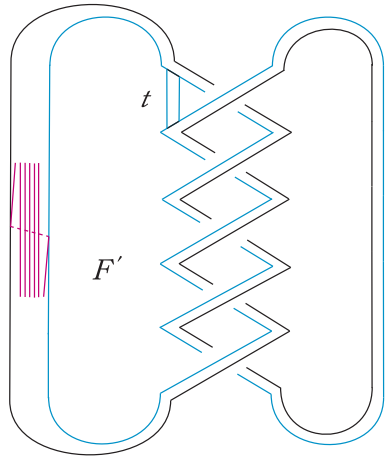}}
    \caption{Genus $t+1$ Heegaard splitting with a strong $(D,F')$
    pair}\label{fig4}
\end{figure}

 In particular, if $K$ is a torus knot, $K(1/q)$ is a
 Seifert fibered space over $S^2$ with three exceptional fibers.
 The splitting induced by an unknotting tunnel is the
 ``vertical" splitting and
 the strongly irreducible non-minimal genus splitting is the
 ``horizontal" splitting \cite{Moriah-Sedgwick}, \cite{Sedgwick}.
 Hence Theorem 1.2 gives some insight to the relation of a vertical
 splitting and a horizontal splitting of such manifolds (Fig. 4)
 (Fig. 5).

\begin{figure}[h]
    \centerline{\includegraphics[width=5cm]{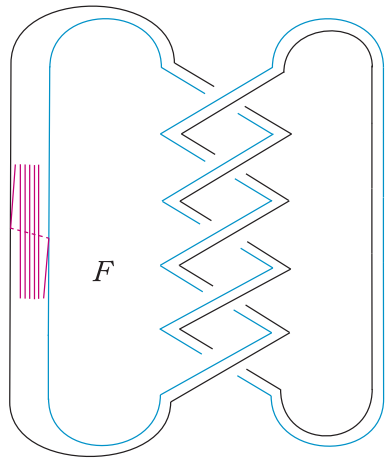}}
    \caption{Strongly irreducible non-minimal genus
    Heegaard splitting of genus $2n$}\label{fig5}
\end{figure}

\end{document}